\documentclass[12pt]{amsart}
\usepackage{amsmath,amsfonts,amstext,amsthm,amssymb,pxfonts}
\usepackage[colorlinks,hypertexnames=false]{hyperref}

\theoremstyle{plain}
\newtheorem{lemma}[equation]{Lemma} 
 
\newtheorem{theorem}[equation]{Theorem}

\theoremstyle{definition}
\newtheorem{definition}[equation]{Definition} 
\newtheorem{example}[equation]{Example}

\theoremstyle{remark}
\newtheorem{remark}[equation]{Remark}

\newcommand\R{\mathbb{R}}

\newcommand\dd{\delta}
\newcommand\ee{\varepsilon}

\allowdisplaybreaks

\numberwithin{equation}{section}

\usepackage{MnSymbol}

\usepackage{mathtools}
\usepackage{rotating}

\title[A simple proof of the sharp weighted estimate . . .]{A simple proof of the sharp weighted estimate for Calderon-Zygmund operators on homogeneous spaces}
\author[T.C.~Anderson]{Theresa C. Anderson}
\author[A.~Vagharshakyan]{Armen Vagharshakyan}

\begin{document}
\maketitle
\begin{abstract}
Here we show that Lerner's method of local mean oscillation gives a simple proof of the $A_2$ conjecture for spaces of homogeneous type: that is, the linear dependence on the $A_2$ norm for weighted $L^2$ Calderon-Zygmund operator estimates.  In the Euclidean case, the result is due to Hyt\"{o}nen, and for geometrically doubling spaces, Nazarov, Rezinikov, and Volberg obtained the linear bound.
\end{abstract}

\section{Introduction}

Weighted norm estimates for singular integral operators in Euclidean space and, more generally, spaces of homogeneous type are classical results. More recently, a variety of methods
have been developed which yield the sharp (linear) dependence of norm bounds for such operators with respect to the Muckenhoupt $A_2$ weight class. In \cite{H}, Hyt\"{o}nen gave the first complete proof of these linear bounds, the so-called 
$A_2$ conjecture, and in \cite{NRV}, the $A_2$ conjecture was proven for geometric doubling spaces by adapting the technology of random dyadic grids.  Subsequently, an extremely elegant and different approach due to Lerner yielded another proof of the conjecture in the Euclidean setting \cite{L2}. 

Spaces of homogeneous type admit a dyadic structure. This was first demonstrated in \cite{C} and later amplified in \cite{HK}, where some additional structure is proven.  For a discussion of homogeneous spaces, see the classical paper by Coifman and Weiss \cite{CW}. 
We show here that Lerner's simple approach to the $A_2$ conjecture also works in the setting of spaces of homogeneous type.  The main tool that is necessary to carry out this approach is the additional dyadic structure developed in \cite{HK}, specifically a version of T. Mei's family of dyadic spaces whereby any ball can be covered by some ``dyadic cube" of roughly the same size. One main contribution of this article is to provide a self-contained, streamlined argument;  all of the essential properties about median and oscillation required to prove this theorem are set forth.  These objects were defined and their properties appear in several papers of Lerner on this and related subjects.  Our objective then is to give a simple proof of this theorem:

\begin{theorem}
\label{maintheorem}
Let $T$ be a Calderon-Zygmund operator and $X$ a homogenous space.  Then for any $w\in A_2$,
\[\|T\|_{L^2(w)}\leq C(T,X)[w]_{A_2}.\]
\end{theorem}

The simplicity of our arguments comes from a version of Lerner's local mean oscillation decomposition.  This concept apppeared in the work of \cite{J} and \cite{St} and has gained popularity through its development in several papers of Lerner \cite{L1},\cite{L2},\cite{L3}.  The technique requires very few assumptions on the measurable space; in section 2 we collect all the properties of homogeneous spaces used in our proof. 

The strategy of the proof is to adapt Lerner's elegant maximal decomposition of a Calderson-Zygmund operator, together with the additonal dyadic structure provided by Mei's lemma, which was shown in \cite{HK}.  The proof we give is completely self contained.  The paper is organized in the following way: relying on few \emph{a priori} assumptions, we begin section 3 with a review of the median and its generalization.  Homogeneous spaces are introduced in section 4, where we adapt a dilation structure to the dyadic grids of honogeneous spaces, leading to our maximal decomposition.  We discretize the operator in section 5.  Finally, section 6 is devoted to proving \ref{maintheorem} by showing a linear $A_2$ bound for an operator from our maximal decomposition.  The structure of the proof is as in Lerner \cite{L2}, but steps 3 and 4 of lemma \ref{discreteoperatornorm} and lemma \ref{howmedianshouldbetreated} are new.  We finally remark that all absolute constants in this paper depend only on the space $X$ and the operator $T$.

Earlier important contributions to this subject can be found in: \cite{Armen}, \cite{B}, \cite{Buckley}, \cite{CMP1}, \cite{CMP2}, \cite{H}, \cite{H2}, \cite{H3}, \cite{JT}, \cite{LPR}, \cite{NRV}, \cite{P1}, \cite{P2}, \cite{PTV}, and \cite{PV}.

\section{Background and definitions}
\begin{definition}\label{homogeneous}
 Following \cite{C} we define a homogeneous space as a triple
$(X, \rho, |\cdot|)$ where $X$ is a set, $\rho$ is a quasimetric, that is:
\[1.\;\rho(x,y)=0 \text{ if and only if } x=y,\] 
\[2. \;\rho(x,y)=\rho(y,x) \text{ fo all } x,y\in X,\] 
\[3. \;\rho(x,z) \leq C_0(\rho(x,y)+\rho(y,z)) \text{ for all }x,y,z\in X\] for some absolute constant  $C_0>0$ (quasitriangle inequality),
and the positive measure $|\cdot|$ is \emph{doubling}:  
\[0<|B(x_0,2r)|\leq C_d|B(x_0,r)|<+\infty\] for some absolute constant $C_d$.
\end{definition}
\begin{definition}\label{calderon}
We'll say that $K:X\times X\setminus\lbrace{x=y\rbrace}\to R$ is a Calderon-Zygmund kernel if there exist $\eta>0$ and $C<\infty$ such that for all $x_0\neq y\in X$ and $x\in X$ it satisfies the decay condition:
\begin{equation}
\label{decay}
|K(x_0,y)|\leq \frac{C}{|B(x_0,\rho(x_0,y))|}
\end{equation}
and the smoothness condition for $\rho(x_0,x)\leq \eta\rho(x_0,y)$:
\begin{equation}
\label{smoothness}
|K(x,y)-K(x_0,y)|\leq\left(\frac{\rho(x,x_0)}{\rho(x_0,y)}\right)^{\eta}\frac{1}{|B(x_0,\rho(x_0,y))|},
\end{equation}
\begin{equation*}
|K(y,x)-K(y,x_0)|\leq\left(\frac{\rho(x,x_0)}{\rho(x_0,y)}\right)^{\eta}\frac{1}{|B(x_0,\rho(x_0,y))|}.
\end{equation*}
\end{definition}
\begin{definition}
Let $T$ be a singular integral operator associated to Calderon-Zygmund kernel $K$. If in addition $T$ is $L^2$ bounded, we say that $T$ is a Calderon-Zygmund operator.
\end{definition}

As mentioned, the $A_2$ theorem on homogeneous spaces involves the extension of the definition of $A_p$ weight classes to include homogeneous situations.
\begin{definition}
Let $w$ be a weight (a nonnegative function locally integrable on X).  We say that $w\in A_2$ if \[\sup_{r>0}\fint_{B(x,r)}\omega(\fint_{B(x,r)}\omega^{-1}) = [\omega]_{A_2} < \infty. \] where the measure is the homogeneous doubling measure.  We call $[\omega]_{A_2}$ is the $A_2$ weight characteristic.
\end{definition}
Note that if $w\in A_2$ (with respect to balls), then $w\in A_2$ with respect to cubes (see remark \ref{a2dyadic}).

Now we list some basic background results for homogeneous spaces that we appeal to later on.  The main construction appears later.
\begin{theorem}
\cite{MS} We can define a new quasimetric, $\rho '$, induced by the topology of the balls of $X$.  If all balls are open, then $\rho$ is equivalent to $\rho '$.
\end{theorem}
\begin{theorem}
\cite{C} The Hardy-Littlewood maximal functional $M(f) = \sup_{B}\fint_B|f(x)|$ for balls $B(x,r)$ is  weak (1,1), strong (2,2), and bounded in $L^2(w)$. 
\end{theorem}
\begin{theorem}
\cite{CW2} Let $T$ be a Calderon-Zygmund operator on a homogeneous space.  Then $T$ is weak $L_1$.
\end{theorem}

\section{Median}
We collect some standard properties of the median and its generalizations.  Note that few assumptions are required on the measure.
Let $(Q, A, |\cdot|)$ be a measure space where $|Q| < \infty$.  Let $\mu$ be the corresponding normalized measure so that $\mu(Q)=1$.
For a measurable function:  
\begin{equation*}
f:Q\rightarrow \R 
\end{equation*}
denote by $F(a)=\mu(f<a)$ its distribution function.
 
\begin{definition}\label{median}
The median of $f$, $m(f)$, is defined as
\[m(f)=\sup\lbrace a>0\colon F(a)\leq 1/2\rbrace.\]
\end{definition}

Now we'll point out several properties of the median that we'll be using.
\begin{lemma}\label{mfisobtained}
We have
\begin{equation*}
\lbrace a\colon F(a)\leq 1/2\rbrace=(-\infty;m(f)],
\end{equation*}
in particular,
\begin{equation*}
F(m(f))\leq 1/2.
\end{equation*}
\end{lemma}
\begin{proof}
The fact that the set $\lbrace a\colon F(a)\leq 1/2\rbrace$ includes its supremum follows from the left continuity of the distribution function $F$.
\end{proof}
\begin{lemma}\label{mfplus}
We have
\begin{equation*}
\mu(f\leq m(f))\geq  1/2.
\end{equation*}
\end{lemma}
\begin{proof}
Observe that $\mu(f\leq m(f))=\lim_{\epsilon\rightarrow 0+}F(m(f)+\epsilon)$ and $F(m(f)+\epsilon)>1/2$ by \ref{median}, thus the claim follows.
\end{proof}
\begin{lemma}\label{monotonicity} If $f\leq g$, then $m(f)\leq m(g).$  
\end{lemma}
\begin{lemma}
\label{abs}
 $|m(f)|\leq m(|f|)$.
\end{lemma}
\begin{proof}
By \eqref{monotonicity} we have. $m(f)\leq m{(|f|)}$.  On the other hand by \eqref{mfisobtained},
\begin{equation*}
\mu(f\leq m(f))\leq 1/2
\end{equation*}
so that $\mu(f> m(f))\leq 1/2.$  Equivalently $\mu(-f< -m(f))\leq 1/2$,
hence by the definition of the median of $-f$, \[-m(f)\leq m(-f).\]  This coupled with \eqref{monotonicity} gives:
\begin{equation*}
-m(f)\leq m(-f)\leq m(|f|)
\end{equation*}
\end{proof}
\begin{lemma}\label{constant}
For any $c\in \mathbb{R}$,
\begin{equation*}
 m(f-c) = m(f)-c
\end{equation*}
\end{lemma}
\begin{lemma}\label{changingconstant}
We have
\begin{equation*}
 m(|f-m(f)|)\leq 2\inf_{c\in R} m(|f-c|)
\end{equation*}
\end{lemma}
\begin{proof}
We have 
\[m(|f-m(f)|)=m(|f-c-m(f-c)|) \leq m(|f-c|+|m(f-c)|) =\]\[= m(|f-c|)+|m(f-c)|\leq 2m(|f-c|)\] where we applied \ref{constant}, \ref{monotonicity}, \ref{constant} again and finally \ref{abs}.
\end{proof}
It is important to note that the triangle inequality fails to hold for the median, as illustrated by the following example.
\begin{example}
Let $Q=[0,1]$ with $|\cdot|$ as Lebesgue measure. Let $f(x) = 1\chi_{[0,3/8]} + \frac{1}{4}\chi_{[3/8,1]}$ and $g(x) = \frac{1}{4}\chi_{[0,5/8]}+1\chi_{[5/8,1]}$.  Then $m(f+g) = 5/4$, but $m(f) = m(g) = 1/4$.
\end{example}
However, we have the following:
\begin{lemma}\label{mediantosup}
\begin{equation*}
m(|f|)\leq \|f-g\|_{\infty}+m(|g|)
\end{equation*}
\end{lemma}
\begin{proof}
\[m(|f|) \leq m(|f-g|+|g|)\leq m(\|f-g\|_{\infty}+|g|) = \|f-g\|_{\infty}+m(|g|)\]
using \ref{monotonicity} twice, followed by \ref{constant}.  Note that we could replace $|f-g|$ with $\|f-g\|_{\infty}$ inside the median since the median does not change upon removing a set of measure 0.
\end{proof}

Lerner defined a generalization of the median value, $\omega_{\lambda}$.
\begin{definition}
For $0<\lambda<1$ and define
\begin{equation*}
 \omega_{\lambda}(f)=\sup(a: F(a)\leq 1-\lambda).
\end{equation*}
\end{definition}
\begin{remark}
The median is a special case: 
 $\omega_{1/2}(f) = m(f)$.
 \end{remark}
If unclear from the context, we'll write $\omega_\lambda(f,Q)$ instead of $\omega_\lambda(f)$ to emphasize the domain of $f$.
\newline
The analogues of \ref{mfisobtained}, \ref{mfplus}, \ref{monotonicity}, \ref{mediantosup} for $\omega_{\lambda}$ hold, for example the analogue of \ref{mediantosup} is:
\begin{lemma}\label{wtosup}
For any $0<\lambda<1$
\begin{equation*}
\omega_{\lambda}(|f|)\leq \|f-g\|_{\infty}+\omega_{\lambda}(|g|)
\end{equation*}
\end{lemma}

The analogue of \ref{changingconstant} holds only for a range of $\lambda$'s.
\begin{lemma}\label{changingconstantw}
For all $0<\lambda \leq 1/2$ we have:
\begin{equation*}
\omega_{\lambda}(|f-m(f)|)\leq 2\inf_{c\in \R} \omega_{\lambda}(|f-c|)
\end{equation*}
\end{lemma}
\begin{proof}
Arguing just like in \ref{changingconstant} we get
\[\omega_{\lambda}(|f-m(f)|) = \omega_{\lambda}(|f-c-m(f-c)|)\leq \omega_{\lambda}(|f-c|+|m(f-c)|)\leq \omega_{\lambda}(|f-c|)+m(|f-c|),\]  but $\omega_{\lambda}(|f-c|)+m(|f-c|)\leq 2\omega_{\lambda}(|f-c|)$ since $m(|f|)\leq \omega_{\lambda}(|f|)$ only for $0\leq\lambda\leq 1/2$.
\end{proof}
We additionally point out the following lemma that will relate $\omega_{\lambda}$ to the maximal operator:
\begin{lemma}
\label{wtomaximal}
For $0<\lambda<1$ we have
\begin{equation*}
\lambda\omega_{\lambda}(|f|,Q) \leq \frac{1}{|Q|}\int_Q |f|
\end{equation*}
\end{lemma}
\begin{proof}
Let
\begin{equation*}
a= \frac{1}{\lambda |Q|}\int_Q |f|.
\end{equation*}
We only need to consider the case $a>0$.

By Chebychev's inequality, we get
\begin{equation*}
\mu\left(|f(x)|\geq a\right)=\frac{|\{x\in Q: |f(x)|\geq a\}|}{|Q|}\leq \frac{1}{a} \cdot\frac{1}{|Q|}\int_Q |f|=\lambda.
\end{equation*}
Consider the two possible cases
\begin{equation*}
\mu(|f(x)|\geq a)< \lambda \text{ and } \mu(|f(x)|\geq a)= \lambda.
\end{equation*}

In the first case, $\mu(|f(x)|<a)>1-\lambda$, hence $\omega_{\lambda} (|f|,Q)\leq a$.

In the second case, when the Chebychev inequality is an equality, the function $|f|$ can take only two values: $0$ and $a$ (up to a set of measure zero), so then $|\{x:|f|=a\}|=\lambda |Q|$. This implies $\omega_{\lambda} (|f|,Q)=a$ and 
\begin{equation*}
\lambda \omega_{\lambda}(|f|,Q) = \frac{1}{|Q|}\int_Q |f|.
\end{equation*}

\end{proof}

The following lemma relates $\omega_{\lambda}$ to the weak $L^{1,\infty}$ norm:
\begin{lemma}
\label{wtoweak}
For any $0<\lambda<1$ we have
\begin{equation*}
\omega_\lambda(|f|)\leq \frac{1}{\lambda |Q|}||f||_{1,\infty}
\end{equation*}
\begin{proof}
Using the definition of $\omega_{\lambda}$, taking the complement of the set involved, we have,
\begin{equation*}
\lambda\leq \mu(|f|\geq \omega_{\lambda}(|f|))=\frac{|\{x\in Q:|f|\geq \omega_{\lambda}(|f|)\}|}{|Q|}
\leq  \frac{||f||_{1,\infty}}{\omega_{\lambda}(|f|)|Q|}
\end{equation*}
\end{proof}
\end{lemma}

The final lemma of this section relates $\omega_{\lambda}$ to the values of function via Lebesgue's differentiation theorem.
\begin{lemma}
\label{LDT}
Let $0<\lambda<1$ and $x_0\in Q_n$ where $Q_n$ is a sequence of sets such that
\[
\int_{Q_n}|f-f(x_0)|\to 0.
\]
Then
\[\omega_{\lambda}(f, Q_n)\rightarrow f(x_0).\]
\end{lemma}
\begin{proof}
By \ref{wtomaximal} we have \[|\omega_{\lambda}(f,Q_n) - f(x_0)| =\omega_{\lambda}(|f - f(x_0)|,Q_n)\leq \frac{1}{\lambda} \fint_{Q_n} |f-f(x_0)|\to 0\] 
\end{proof}

\section{Maximal Decomposition}
We prove the analogue of Theorem 2.3 in \cite{L2} for homogeneous spaces. The following theorem describes dyadic decomposition for homogeneous spaces:
\begin{theorem}
\label{dyadic}
\cite{C} and \cite{HK} There exist absolute constants $C>0$, $\delta>0$, $0<\epsilon<1$, a family of sets  $S=\cup_{k\in \mathbb{Z}}S_k$  (called a dyadic decomposition of $X$) and a corresponding family of points $\lbrace x_c(Q)\rbrace_{Q\in S}$ that satisfy the following properties:
\begin{equation}
\label{p1}
\left|X\setminus \bigcup_{Q\in S_k}Q\right| = 0 \text{, for all } k\in \mathbb{Z}
\end{equation}
\begin{equation}
\label{p2}
\text{If } Q_1,Q_2\in S\text{ then either }Q_1\cap Q_2 = \emptyset \text{ or one is contained in the other}.
\end{equation}
\begin{equation}
\label{p3}
\text{ For any }Q_1\in S_k\text{ there exists at least one }Q_2\in S_{k+1}\text{ so that }Q_2\subset Q_1 \text{ (called a child of }Q_1\text{) }
\end{equation}
\begin{equation*}
\text{ and there exists exactly one }Q_3\in S_{k-1}\text{ so that }Q_1\subset Q_3  \text{ (called the parent of }Q_1\text{) }.
\end{equation*}
\begin{equation}
\label{p4}
\text{If } Q_2 \text{ is a child of } Q_1 \text{ then }  |Q_2|\geq \epsilon |Q_1|.
\end{equation}
\begin{equation}
\label{p5}
B(x_c(Q), \delta^k)\subset Q \subset B(x_c(Q), C\delta^k)
\end{equation}
\end{theorem}
\begin{definition}\label{level}
For a $Q\in S$ we'll denote by $k(Q)\in Z$ the only integer for which $Q\in S_{k(Q)}$.
\end{definition}
\begin{definition}\label{localized}
For a.e. $x\in X$ for any $k\in Z$ there exists a unique set denoted by $Q_k(x)$ for whom $x\in Q_k(x)\in S_k$.
\end{definition}

We emphasize that though much of the dyadic terminology is the same for homogeneous spaces and with Euclidean space, the usual concepts of centers and side lengths, and dilations of cubes do not exist as in Euclidean spaces.
\begin{definition}\label{dylation}
Let $S$ be a dyadic decomposition of the space $X$ and $Q\in S_k$. For a number $\lambda>1$ we'll denote by $\lambda Q$ the "dilated" set 
\begin{equation*}
\lambda Q=B(x_c(Q), \lambda C \delta^k),
\end{equation*}
where the constant $C$ is the one provided by \ref{dyadic}.
The set $\lambda Q$ will play the role of the dilation of the set $Q$ by the number $\lambda$. However, note that even in the Euclidean case, $\lambda Q$ is not necesarilty a dilation with respect to the center of $Q$. \end{definition}
\begin{remark}\label{a2dyadic}
From \ref{p5} and the fact that the measure is doubling it follows that if $w$ is an $A_2$ weight then for any $Q\in S$,
\begin{equation*}
\fint_Q \omega\fint_Q \omega^{-1}\leq D_{0} [\omega]_{A_2}
\end{equation*}
where $D_{0}$ is an absolute constant.
\end{remark}
Let $S$ be a dyadic decomposition of $X$.
Fix $Q_0\in  S$ and a measurable function $f:Q_0\rightarrow \R$. In the rest of this section we'll construct the maximal decomposition of $f$ (stated in \ref{maximaldecomposition}).
For convenience denote $f_{m}=|f-m(f,Q_0)|$.
\begin{definition}
Denote by $M$ those elements $Q$ of $S$ that lie in $Q_0$ and that are maximal  under inclusion with respect to the following property:
\begin{equation*}
m(f_{m},Q)>\omega_{\ee/4}(f_{m},Q_0),
\end{equation*}
(where $0<\epsilon<1$ is the number provided by theorem \ref{dyadic}).
\end{definition}
Let $\hat{M}$ be the set of the maximal elements under inclusion of the set of parents of $M$. Note that the elements of $\hat{M}$ are disjoint and for $Q\in \hat{M}$:
\begin{equation}\label{prop2}
m(f_{m},Q)\leq\omega_{\ee/4}(f_{m},Q_0) 
\end{equation}
Also,
\begin{lemma}
We have the following:
\begin{equation}
\label{prop1}
\sum_{Q\in \hat{M}}|Q|\leq \frac{|Q_0|}{2}
\end{equation}
\end{lemma}
\begin{proof}
For $Q\in M$ we have that $m(f_{m},Q)>\omega_{\ee/4}(f_{m},Q_0)$, so that
\begin{equation*}
\frac{|Q|}{2}\leq |\lbrace x\in Q\colon f_{m}(x)\geq m(f_{m},Q)     \rbrace|   \leq    |\lbrace x\in Q\colon f_{m}(x)> \omega_{\ee/4}(f_{m},Q_0)     \rbrace|.
\end{equation*}
where the first inequality is by \ref{mfplus}.
The elements of $M$ are disjoint, so summing over these elements we get
\begin{equation*}
\frac{1}{2}\sum_{Q\in M}|Q|\leq  |\lbrace x\in Q_0\colon f_{m}(x)> \omega_{\ee/4}(f_{m},Q_0)   \rbrace|\leq \ee\frac{|Q_0|}{4}
\end{equation*}
where the second inequality is due to the analogue of \ref{mfplus} for $\omega_{\lambda}$.
By \ref{p4},
 \[\frac{\ee}{2}\sum_{Q\in \hat{M}}|Q|\leq \frac{1}{2}\sum_{Q\in M}|Q|\leq\frac{\ee|Q_0|}{4}.\]
\end{proof}
\begin{remark}
Note that  \ref{prop1} implies that $Q_0\notin \hat{M}$. 
\end{remark}
\begin{remark}
\label{decompositionstart}
For any $x\in Q_0$ we have the following decomposition 
\begin{equation*}
f(x)-m(f,Q_0)=(f(x)-m(f,Q_0))\chi_{Q_0\setminus\cup_{Q\in \hat{M}}Q}(x)+\end{equation*}\begin{equation*}+\sum_{Q\in \hat{M}}\left(f(x)-m(f,Q)\right)\chi_{Q}(x)+\sum_{Q\in \hat{M}}(m(f,Q)-m(f,Q_0))\chi_{Q}. 
\end{equation*}
\end{remark}
We'll be estimating parts of that decomposition. We first need the following results.
\begin{theorem}
\cite{T} We call $x$ a \emph{Lebesgue point} if $\fint_{B(x,r)}|f(x)|\to f(x)$ as $|B|\to 0$ (note that $|\cdot |$ is the homogeneous measure).  If $f\in L^1(X)$ then almost every point is Lebesgue.
\end{theorem}
\begin{lemma}
\label{LDT2}
Let $x\in Q_0$ be a Lebesgue point of $f$,  then
 \[\frac{1}{|Q_k(x)|}\int_{Q_k(x)}|f-f(x)|\rightarrow 0, \quad \text{ as } k\rightarrow +\infty,\] 
 where $Q_k(x)$ is defined in \ref{localized}.
\end{lemma}
\begin{proof}
To prove this, we use a few key facts.  If $k$ is sufficiently large then $Q_k(x)\subset Q_0$. Since
\begin{equation*}
B(x_c(Q_k),\dd ^k)\subseteq Q_k(x)\subseteq B(x_c(Q_k),C\dd^k),\end{equation*} by applying the doubling property, we can arrive at the result.
\end{proof}
The following lemma estimates the first term in \ref{decompositionstart}.
\begin{lemma}\label{firsttermest}
For a.e. $x\in Q_0\setminus\cup_{Q\in \hat{M}}Q$
\begin{equation*}
|f(x)-m(f,Q_0)|\leq \omega_{\ee/4}(f_{m},Q_0) 
\end{equation*}
\end{lemma}
\begin{proof}
For an $x\in Q_0\setminus\cup_{Q\in \hat{M}}Q$, we have that $f_m(x)>\omega_{\ee/4}(f_{m},Q_0)$.  Since $x$ is a Lebesgue point for $f_m$, then by \ref{LDT2}, \[m(f_m,Q_{k}(x))>\omega_{\ee/4}(f_{m},Q_0) \text{ for a sufficiently large }k,\]  implying that $Q_k(x)$ is a subset of an element of $M$, thus a subset of an element of $\hat{M}$.  In particular, $x\in \cup_{Q\in \hat{M}}Q$, a contradiction.
\end{proof}
\begin{remark}\label{secondtermest}
Also, using \ref{prop2} we get that for any $Q\in \hat{M}$: \[|m(f,Q)-m(f,Q_0)|=|m(f-m(f,Q_0),Q)|\leq m(f_{m},Q)\leq \omega_{\ee/4}(f_{m},Q_0)\]
\end{remark}
\begin{remark}\label{beforeinduction}
Using the  estimates \ref{secondtermest}, \ref{firsttermest} in \ref{decompositionstart} we get that for a.e. $x\in Q_0$:
\begin{equation*}
|f(x)-m(f,Q_0)|\leq M^{\#}_{\ee/4,Q_0}(f)(x)\chi_{Q_0\setminus\cup_{Q\in \hat{M}}Q}(x)+\omega_{\ee/4}(|f-m(f,Q_0)|,Q_0)\chi_{Q_0}(x)+
\sum_{Q\in \hat{M}}|f-m(f,Q)|\chi_{Q}(x),
\end{equation*}
where
 \[M^{\#}_{\lambda,Q_0}(f)(x) = \sup\lbrace \omega_{\lambda}(|f-m(f,Q)|,Q)\colon\; x\in Q\subset Q_0, Q\in S   \rbrace,\] following Lerner's notation.
\end{remark}
\begin{remark}\label{maximaldecomposition}
Applying the estimate of \ref{beforeinduction} inductively (to the last term in the sum) and using \ref{prop1} we get that for  a.e. $x\in Q_0$:
\begin{equation*}
|f(x)-m(f,Q_0)|\leq M^{\#}_{\ee/4,Q_0}(f)(x)
+\sum_{Q\in S(Q_0)}   \omega_{\ee/4}(|f-m(f,Q)|,Q)  \chi_Q(x),
\end{equation*}
where the sum is spread over a family of sets $S(Q_0)$ which is sparse in $Q_0$ with respect to the dyadic decomposition $S$, that is \begin{equation*}S(Q_0)=\bigcup_{n=0,1,\dots}C_n\end{equation*} (where $n$ represents the step in the induction), so that:
\newline
1. Each element of $S(Q_0)$ belongs to $S$ and is a subset of $Q_0$.
\newline
2. The elements of each family $C_n$ are disjoint.
\newline
2. $C_0=\lbrace Q_0 \rbrace$.
\newline
3. If $n>0$ then each $Q\in C_n$ is a subset of an element of $C_{n-1}$
\newline
5. For any $Q_1\in C_n$ we have that: \begin{equation}\label{half} \left|\bigcup_{Q\in C_{n+1}} Q\cap Q_l\right|\leq \frac{|Q_1|}{2}.\end{equation}
Note how this relates to \ref{prop1}.
\end{remark}
\section{Discretization of Operator}
In this section, we follow Lerner's strategy in \cite{L2} in order to reduce estimates of a Calderon-Zygmund operator to the case of a discrete operator.
\begin{lemma}\label{linearization}
Let $T$ be a Calderon-Zygmund operator (see \ref{calderon}), $Q\in S$ be an element of a dyadic decomposition $S$ of the space $X$ (see \ref{dyadic}) and $0<\epsilon<1$ be a constant provided by \ref{dyadic}. We then have the following:
\begin{equation*}
\omega_{\ee/4} (|Tf-m(Tf,Q)|,Q)\leq D_1 \sum_{l=1}^{\infty} \frac{1}{2^{l\eta}}\fint_{2^l Q} |f(y)|dy,
\end{equation*}
where $2^l Q$ is defined in \ref{dylation} and $D_1$ is an absolute constant.
\end{lemma}
\begin{proof}
For compactness denote $x_c=x_c(Q)$ (see \ref{localized}) and $k=k(Q)$ (see \ref{level}).   By \ref{dylation}, define  $Q^*=\frac{1}{\eta}Q=B(x_c,\frac{C}{\eta}\delta^k)$where $\eta$ is the constant provided by the smoothness condition for $T$ (see \ref{calderon}). For a function $f$ denote
\begin{equation*}
f_1=f\cdot \chi_{Q^*} \text{ and } f_2=f\cdot\chi_{X\setminus Q^*}
\end{equation*}
By \ref{changingconstantw} and \ref{wtosup} we have
\begin{equation}\label{woperator}
\omega_{\ee/4} (|Tf-m(Tf,Q)|,Q) \leq  2\omega_{\ee/4} (| Tf-(Tf_2)(x_c)|,Q) \end{equation}\begin{equation*}\leq 2\omega_{\ee/4} (| Tf_1+Tf_2-(Tf_2)(x_c)|,Q)\leq 
2\omega_{\ee/4}(|Tf_1|,Q)
 +2\|(Tf_2)(x)-(Tf_2)(x_c)\|_{L_\infty(Q)}
\end{equation*}
For the first term in the right side of \ref{woperator} note that the operator $T$ is of weak type $(1,1)$, thus by \ref{wtoweak}.:
\begin{equation*}
\omega_{\ee/4}(|Tf_1|,Q)\leq \frac{4||Tf_1||_{1,\infty}}{\epsilon |Q|}\leq\frac{4||T||_{1,\infty}}{\epsilon}\cdot\frac{||f_1||_1}{|Q|}\leq \frac{4||T||_{1,\infty}}{\epsilon}\cdot \frac{|Q^*|} {|Q|}\fint_{Q^*}|f|
\end{equation*}
Note that by \ref{p5},
\[
\frac{|Q^*|}{|Q|}\leq \frac{|B(x_c,C\delta^k/\eta)|}{|B(x_c,\delta^k)|},
\]
which is bounded by an absolute constant due to the doubling property.
\newline
In second term in the right side of \ref{woperator}, for $x\in Q$ we have:
\begin{equation*}
\rho(x,x_c)< C\delta^k
\end{equation*}
for $y\in X\setminus Q^*$ we have:
\begin{equation}\label{farfar}
\frac{C}{\eta}\delta^k\leq  \rho(x_c,y),
\end{equation}
and due to \ref{farfar} there exists unique $l> 1$ so that:
\begin{equation*}
2^{l-1}\dd^k\leq \rho(x_c,y)\leq 2^l \dd^k.
\end{equation*}
We can then apply the kernel estimate
\begin{eqnarray*}
|K(x,y)-K(x_c,y)|&\leq & \left(\frac{\rho(x,x_c)}{\rho(x_c,y)}\right)^{\eta}\frac{1}{|B(x_c,\rho(x_c,y))|}\\&\leq
&\frac{(2C)^{\eta}}{2^{l\eta}}\cdot\frac{1}{|B(x_c,2^{l-1}\dd^k)|}\chi_{B(x_c,2^l\dd^k)}(y)\\&\leq &\frac{(2C)^{\eta}}{2^{l\eta}} \cdot  \frac{|B(x_c,2^{l}C\delta^k)|}{|B(x_c, 2^{l-1}\delta^k)|}\cdot\frac{1}{|2^ lQ|}\chi_{2^lQ}(y),
\end{eqnarray*}
where the second factor is bounded by an absolute constant due to the doubling property. Thus
\begin{eqnarray*}
|(Tf_2)(x)-(Tf_2)(x_c)|&\leq &\int_{X\setminus Q^*} |f_2(y)||K(x,y)-K(x_c,y)|dy \\&\leq &D_1\sum_{l=1}^{\infty} \frac{1}{2^{l\eta}|2^l Q|}\int_{X\setminus Q^*} |f_2(y)|\chi_{2^l Q}(y) dy\\ & = &D_1\sum_{l=1}^{\infty} \frac{1}{2^{l\eta}}\fint_{2^l Q} |f(y)|dy,
\end{eqnarray*}
for some absolute constant $D_1$.
\end{proof}

Let $T$ be a Calderon-Zygmund operator (see \ref{calderon}). Let $S$ be a dyadic decomposition of the space $X$ (see \ref{dyadic}). Fix an arbitrary $Q_0\in S$.
By \eqref{maximaldecomposition}  for a.e. $x\in Q_0$ we have:
\begin{equation*}
|(Tf)(x)-m(Tf,Q_0)|\leq M^{\#}_{\ee/4,Q_0}(Tf)(x)+\sum_{Q\in S(Q_0)} \omega_{\ee/4}(|Tf-m(Tf,Q)|,Q)\chi_{Q}(x).
\end{equation*}
By \eqref{linearization}
\begin{equation*}
|(Tf)(x)|\leq |m(Tf,Q_0)|+D_2\cdot M(f)(x)+D_1\sum_{k=1}^{\infty} \frac{1}{2^{k\eta}} \sum_{Q\in S(Q_0)} \fint_{2^k Q}|f|\chi_{Q}(x),
\end{equation*}
where $D_2$ is a certain constant depending on the operator $T$ (on the parameter $\tau$ to be precise).
\newline
Let $\omega$ be an $A_2$ weight.  Then
\begin{equation*}
||Tf||_{L_2(\omega,Q_0)}\leq D_2\cdot||M(f)||_{L_2(\omega,Q_0)}+D_1\cdot\sum_{k=1}^{\infty} \frac{1}{2^{k\eta}} \left|\left|\sum_{Q\in S(Q_0)} \fint_{2^k Q}|f|\chi_{Q}\right|\right|_{L_2(\omega)}+
||m(Tf,Q_0)||_{L_2(\omega,Q_0)}
\end{equation*}
We'll now estimate the last term. By \ref{wtoweak}:
\begin{equation*}
|m(Tf,Q_0)|\leq \frac{||Tf||_{L_{1,\infty}(Q_0)}}{|Q_0|}\leq ||T||_{1,\infty} \frac{||f||_{L_{1}(Q_0)}}{|Q_0|}
\end{equation*}
Applying Holder's inequality and recalling that  $[\omega]_{A_2}\geq 1$, we get:
\begin{equation}\label{lernermissing}
||m(Tf,Q_0)||_{L_2(w,Q_0)}\leq \frac{||T||_{1,\infty}}{|Q_0|}\left(\int_{Q_0}\omega\right)^{1/2}||f||_{L_{1}(Q_0)}\leq
 \end{equation}
 \begin{equation*}
||T||_{1,\infty} ||f||_{L_2(\omega,Q_0)}\left(\int_{Q_0}{\omega^{-1}}\right)^{1/2}\left(\int_{Q_0}{\omega}\right)^{1/2}\frac{1}{|Q_0|}=||T||_{1,\infty}||f||_{L_2(\omega,Q_0)} [\omega]_{A_2}^{1/2}\leq ||T||_{1,\infty}||f||_{L_2(\omega,Q_0)} [\omega]_{A_2}.
\end{equation*}

Thus, in order to prove \ref{maintheorem} (i.e. the linear $A_2$ bound for $T$) if suffices to show a linear $A_2$ bound for the discrete operator:
\begin{equation}\label{discreteoperator}
A_k(f)=\sum_{Q\in S(Q_0)} \fint_{2^k Q}|f|\chi_{Q},\quad k=1,2,\dots
\end{equation}
with a norm estimate depending linearly on the parameter $k$ ($k$ is called the complexity of the discrete operator $A_k$).  
This will be proved in the next section.
\indent

\section{Linear bound with respect to complexity of discrete operators}
By Theorem 4.1 in  \cite{HK} there exist  constants $J\in N$ and $D>0$ depending only the space $X$ and there exists a corresponding collection of dyadic decompositions $\lbrace S^j\colon j=1,2\dots, J\rbrace$ of the space $X$ such that for every ball $B(x,r)$ there exists some $1\leq j\leq J$ and $Q\in S^j$ with $B(x,r)\subset Q$ and $diam(Q)\leq Dr$.
 \newline
 Hence in order to prove a linear $A_2$ bound for the operator \ref{discreteoperator} it suffices to prove a linear $A_2$ bound, depending linearly on the complexity $k$, for the operators of the following form:
\begin{equation}\label{bidefinition}
B_{j,k}(f)=\sum_{Q\in S^{\prime}(Q_0)} \frac{1}{|Q^*|}\int_{Q^*}f \cdot\chi_Q\quad j=1,2,\dots, J
\end{equation}
The sum in the definition of $B_{j,k}$ is spread over $S^{\prime}(Q_0)$ which is a certain subset
of $S(Q_0)$ which has the following property:  for any $Q\in S^{\prime}(Q_0)$  there exists a set in $S^j$ denoted by $Q^*\in S^j$ so that 
\begin{equation}\label{mei1}
2^k Q=B(x_c(Q),2^k C\delta^{k(Q)})\subset Q^*
\end{equation}
 and 
 \begin{equation}\label{mei2}
 diam(Q^*)\leq D \cdot 2^k C \delta^{k(Q)},
 \end{equation} 
 where $k(Q)$ was defined in \ref{level}.  Thus $S'(Q_0) = \cup C_n'$ where some levels may be empty.
\newline
In the future we'll write $B$ instead of $B_{j,k}$, supressing the dependence on the parameters $j$ and $k$.
We'll also need the formal adjoint of $B$:
\begin{equation*}
B^*(f)=\sum_{Q\in S^{\prime}(Q_0)} \frac{1}{|Q^{\ast}|} \int_{Q} f \cdot\chi_{Q^{\ast}}
\end{equation*}
The rest of the paper will be devoted to the proof of a linear $A_2$ bound of the discrete operator $B$ depending linearly on the complexity $k$, proved in \ref{biitself}, thus completing the proof of \ref{maintheorem}.  We follow the outline from Lerner, but emphasize that due to the design of homogenous spaces, our argument in steps 3 and 4 of \ref{discreteoperatornorm} is different.  In \ref{howmedianshouldbetreated}, we give an argument for an estimate which is implicit, but not stated in \cite{L2}.
\begin{lemma}\label{eltwo}
For a constant $\gamma>0$ depending on the space $X$ only,
\begin{equation*}
||B(f)||_2=||B^*(f)||_2\leq \gamma||f||_2
\end{equation*}
\end{lemma}
\begin{proof}
For any $Q\in C_n'$ denote $E(Q)=Q \setminus \bigcup_{Q_1\in C_{n+1}} Q_1$. The sets $E(Q)$ are pairwise disjoint, and by \ref{half} we have: $|E(Q)|\leq |Q|/2$,  thus for any $f,g\in L_2(X)$:
\begin{equation*}
\int_X B(f)\cdot g\leq \sum_{Q\in S^{\prime}(Q_0)} \fint_{Q^*} |f|\cdot \fint_{Q} |g| \cdot |Q|  \leq\end{equation*}\begin{equation*} 2\sum_{Q\in S^{\prime}(Q_0)} \fint_{Q^*} |f|\cdot \fint_{Q} |g| \cdot |E(Q)|\leq 
2\sum_{Q\in S^{\prime} (Q_0)}\int_{E(Q)} (Mf\cdot Mg)\leq 2\int_{X} (Mf\cdot Mg)\leq \gamma ||f||_2||g||_2
\end{equation*}
which finishes the proof.
\end{proof}
\begin{lemma}\label{discreteoperatornorm}
There exists an absolute constant $\beta$ so that we have:
\begin{equation*}
||B^*(f)||_{1,\infty}\leq \beta k||f||_1,
\end{equation*}
where $B=B_{j,k}$, for all $1\leq j\leq J$ and $k\in N$.
\end{lemma}
\begin{proof}
\textbf{Step 1.}
\newline
The set $\Omega=\lbrace (Mf)(x)>a \rbrace$ is open and bounded (since $|Q|<\infty$). Denote by $\lbrace W_i\rbrace_{i\in I}$ the family of elements $W_i$ of the dyadic decomposition $S$ who lie in the set $\Omega$ and who are maximal with respect to the following property $diam(W_i)\leq dist(W_i,\Omega^c)$. The sets $W_i$ are disjoint by maximality. Additionally $\Omega=\cup_i W_i$.
\newline

Indeed, for a point $x\in \Omega$ consider the sets $Q_k(x)$ introduced in \ref{localized}. We have that $diam(Q_k(x))\rightarrow 0$ and by quasitriangle inequality $dist(Q_i(x),\Omega^c)\geq dist(x,\Omega^c)/C_0$ as $k\rightarrow+\infty$. Thus $Q_k(x)\subset \Omega$ and $diam(Q_k(x))<dist(Q_k(x),\Omega^c)$ for a sufficiently large $k\in Z$.
\newline
\textbf{Step 2.}
\newline
Set
\begin{equation*}
b_i=\left(f-\fint_{W_i} f\right)  \chi_{W_i}, \quad b=\sum_{i\in I} b_i, \quad g=f-b
\end{equation*}
For any $a>0$ we have:
\begin{equation}\label{standard}
|\lbrace x\colon |(B^*f)(x)|>a\rbrace|\leq |\Omega|+|\lbrace x\colon |(B^*g)(x)|>a/2\rbrace|+|\lbrace x\in \Omega^c \colon  |(B^*b)(x)|>a/2\rbrace|.
\end{equation}
In order to prove \ref{discreteoperatornorm} we estimate the terms on the right of \ref{standard}. For the first term we use that the maximal function is weak $L_1$, thus:
\begin{equation*}
|\Omega|\leq \frac{D_2}{a}||f||_1,
\end{equation*}
for an absolute constant $D_2$.
\newline
\textbf{Step 3.}
\newline
For the second term we first note that 
\begin{equation}\label{giselone}
g(x)\leq D_3\cdot a
\end{equation}
 for some absolute constant $D_3>0$. Indeed, if $x\in \Omega^c$ then $|g(x)|=|f(x)|\leq a$. 
 \newline
 Now if $x\in \Omega$ then
$x\in W_i$ for some $i\in I$. Denote by $W_i^p$  the parent of $W_i$ in the dyadic decomposition $S$. Two cases are possible:
\newline
\textbf{Case 1.}
$
diam(W_i^p)>dist(W_i^p,\Omega^c).
\newline
$\textbf{Case 2.}
$W_i^p$  contains a point in $\Omega^c$.
\newline
Consider case 1. There exist $y\in \Omega^c$ and $x_1\in W_i^p$ so that
\begin{equation*}
diam(W_i^p)>\rho(x_1,y).
\end{equation*}
Let $x_2$ be an arbitrary point in $W_i$. By quasitriangle inequality:
\begin{equation}\label{yisnotfar}
\rho(x_2,y)\leq C_0(\rho(x_2,x_1)+\rho(x_1,y))< 2C_0 diam(W_i^p)
\end{equation}
Hence
\begin{equation}\label{b2}
W_i\subset B(y,2C_0diam(W_i^p))=:B
\end{equation}
But  
 $y\in \Omega^c$ so that $(Mf)(y)\leq a$.
 Hence, we can estimate:
 \begin{equation*}
|g(x)|=\left |\fint_{W_i} f \right|\leq \frac{1}{|W_i|}\int_{B}|f|\leq \frac{|B|}{|W_i|}\fint_{B}|f|\leq \end{equation*}\begin{equation*}\leq    \frac{|B|}{|W_i|}    (Mf)(y)\leq     \frac{|B|}{|W_i|}    a
\end{equation*}
In order to estimate the fraction $|B|/|W_i|$, take $x_2$ to be $x_c(W_i)$ in \ref{yisnotfar} to claim:
\begin{equation*}
\rho(x_c(W_i),y)<2C_0diam(W_i^p)
\end{equation*}
which together with \ref{b2} implies that
\[\rho(x_c(W_i^p),2C_0diam(W_i^p))\leq C_0(\rho(y,x_c)+\rho(y,2C_0diam(W_i^p)))\leq 2C_0(2C_0diam(W_i^p)\]
so it follows that
\begin{equation*}
B\subset B(x_c(W_i),4C_0^2 diam(W_i^p)).
\end{equation*}
So that we can estimate.\begin{equation*}
\frac{| B|}{|W_i|   }\leq \frac{|B(x_c(W_i),4C_0^2 diam(W_i^p))|}{|W_i|}\end{equation*}
But due to  \ref{p5}  that quantity
 is bounded by an absolute constant by the doubling property. 
\newline
Note that in case 2 it is still true that for a certain $y\in \Omega^c$ we have $(Mf)(y)\leq a$ and $W_i\subset B$ so that case 2 can be treated similarly.
\newline
Now we are ready to estimate the second term of \ref{standard}. Use \ref{eltwo} and \ref{giselone} to write:
\begin{equation*}
|\lbrace x\colon |(B^*g)(x)|>a/2|\leq \frac{4}{a^2}||B^*(g)||_2^2\leq \frac{4\gamma}{a^2}||g||_2^2\leq \frac{4\gamma D_3}{a}||g||_1
\leq \frac{4\gamma D_3}{a}||f||_1\end{equation*}
\textbf{Step 4.}
In order to estimate the third term of \ref{standard} fix $x\in \Omega^c$ and consider:
\begin{equation}
\label{Bstarsum}
(B^* b)(x)=\sum_{i\in l}\sum_{Q}       \left(  \frac{1}{|Q^*|}    \int_Q b_i \right) \cdot   \chi_{Q^*}(x) 
\end{equation}
Assume that for certain $Q$ and $W_i$ we have that: 
\begin{equation*}
 \left( \int_Q b_i \right)\cdot \chi_{Q^*}(x)\neq 0
\end{equation*}
Then $Q\cap W_i\neq\emptyset$. But also $W_i\not\subset Q$ as $\int_{W_i}b_i=0$ by definition. Finally both $W_i$ and $Q$ belong to the same dyadic decomposition $S$. Thus we must have that $Q\subset W_i$. Also $\chi_{Q^*}(x)\neq 0$ implies $x\in Q^*$. Using all of these and \ref{mei1}, \ref{mei2} we can estimate:
\begin{equation}\label{estone}
diam(W_i)\leq dist(W_i, \Omega^c)\leq dist(x_c(Q),x)\leq diam(Q^*)\leq D\cdot 2^{k}C\delta^{k(Q)}\leq 2^{k}\cdot DC\cdot diam(Q)
\end{equation}
Recall that by \ref{level} $k(W_i)$ denotes the only integer for whom $W_i\in S_{k(W_i)}$. We have:
\begin{equation}\label{esttwo}
diam(W_i)\geq  \delta^{k(W_i)}\leq 2^k\cdot DCdiam(Q)\leq 2^k\cdot DC\cdot 2C_0C \delta^{k(Q)}
\end{equation}
We have that for some absolute constants $D_4,D_5>0$,
\begin{equation*}
k(W_i)\leq k(Q)<k(W_i)+D_4\cdot k+D_5
\end{equation*}
where the left hand side follows from $Q\subset W_i$ and right hand side follows from plugging \ref{esttwo} into \ref{estone}. 
Thus, for our fixed $x\in \Omega^c$ we can write
\begin{equation*}
(B^* b)(x)\leq \sum_{i\in l}\sum_{Q\;\colon\; k(W_i)\leq k(Q)\leq D_4 k+k(W_i)+D_5}       \left(  \frac{1}{|Q^*|}    \int_Q |b_i| \right) \cdot   \chi_{Q^*}(x), 
\end{equation*}
but the sets $Q$ for which the number $k(Q)$ is the same are disjoint, thus
\begin{equation*}
\int_{\Omega^c}|(B^*b)(x)|\leq
\sum_{i\in I} \sum_{Q\;\colon\; k(W_i)\leq k(Q)\leq D_4 k+k(W_i)+D_5} \left( \int_Q |b_l|\right)\leq 
\end{equation*}
\begin{equation*}
\sum_{i\in I} (D_4 k+D_5) \int_{W_l} |b_l| \leq  2(D_4 k+D_5)||f||_1
\end{equation*}
So that we can estimate the third term in \ref{standard} as follows:
\begin{equation*}
|\lbrace x\in \Omega^c \colon  |(B^*b)(x)|>a/2\rbrace|\leq \frac{2}{a}|| B^*b ||_{L_1(\Omega^c)}\leq  \frac{2(D_4+D_5)k||f||_1}{a}
\end{equation*}
\end{proof}
\begin{lemma}\label{biw}
For any $Q_1\in S^j$ we have that:
\begin{equation*}
\omega_{\epsilon/4} (|B^*(f)-m(B^*(f))|,Q_1)\leq D_6 k\fint_{Q_1}|f|,
\end{equation*}
where $\epsilon$ is the constant provided by \ref{dyadic}, $B=B_{k,j}$ and $D_{6}$ is an absolute constant.
\end{lemma}
\begin{proof}
Note that the function
\begin{equation*}
\sum_{Q\colon Q\in S^{\prime}(Q_0), Q_1\subsetneq Q^* } \left(\frac{1}{|Q^*|} \int_Q f\right) \chi_{Q^*}(x)
\end{equation*}
is constant on $Q_1$ as both $Q_1$ and the sets $Q^*$ involved in the sum are in $S^j$. Thus, using \ref{changingconstantw} we can write:
\begin{eqnarray*}
\omega_{\epsilon/4} (|B^*(f)-m(B^*(f))|,Q_1)&\leq &\omega_{\epsilon/4} \left(\left|B^*(f)-    \sum_{Q\colon Q\in S^{\prime}(Q_0), Q_1\subsetneq Q^* } \left(\frac{1}{|Q^*|} \int_Q f\right) \chi_{Q^*}(x)
    \right|,Q_1\right)\\&=&
    \omega_{\epsilon/4} \left(\left|\sum_{Q\colon Q\in S^{\prime}(Q_0),Q^*\subset Q_1 } \left(\frac{1}{|Q^*|} \int_Q f\right) \chi_{Q^*}(x)
    \right|,Q_1\right)\\&=&\omega_{\epsilon/4}\left( | B^*(f\cdot \chi_{Q_1})| ,Q_1\right)\leq \omega_{\epsilon/4}\left(  B^*(|f|\cdot \chi_{Q_1}) ,Q_1\right)
        \end{eqnarray*}
        Further by \ref{wtoweak} and \ref{discreteoperatornorm} we have
        \begin{equation*}
        \omega_{\epsilon/4}\left(  B^*(|f|\cdot \chi_{Q_1}) ,Q_1\right)\leq \frac{4}{\epsilon |Q_1|} ||B^*(|f|\cdot \chi_{Q_1})||_{L_{1,\infty}(Q_1)}\leq \frac{4}{\epsilon |Q_1|} \beta k ||f||_{L_1(Q_1)}\leq     \frac{4\beta}{\epsilon} k \fint_{Q_1} |f|   \end{equation*}
        \end{proof}

\begin{lemma}\label{bistar}
\begin{equation*}
||B^*(f)||_{L_2(w)}\leq D_7\cdot k [\omega]_{A_2} ||f||_{L_2(\omega)},
\end{equation*}
where $B=B_{k,j}$ and $D_7$ is an absolute constant.
\end{lemma}
\begin{proof}
Let $Q_2\in S^j$ be a set satisfying:
\begin{equation*}
diam(Q_2)>D\cdot 2^kC\delta^{k(Q_0)},
\end{equation*}
where $Q_0$ is the set involved in the definition of  $B$ (see \ref{bidefinition}). (This can always be achieved by taking $Q_2$ from a sufficiently high level of the dyadic decomposition $S_j$). Then for any $Q^*$ from \ref{bidefinition} using \ref{mei2} we have that $diam(Q^*)< diam(Q_2)$ and hence
\begin{equation}\label{hahahaempty}
Q_2\not\subset Q^*
\end{equation}
 Apply the maximal decomposition \ref{maximaldecomposition}  for the function $B^*(f)$ taking the $Q_0$ of \ref{maximaldecomposition} to be $Q_2$ and apply the estimates obtained in \ref{biw} to estimate:
\begin{equation*}
|(B^*f)(x)|\leq |m(B^*f,Q_2)|+D_6\cdot k \cdot(Mf)(x)+D_6\cdot  k \cdot\sum_{Q_1\in S^j(Q_2)}\fint_{Q_1}|f|\quad \text{ for a.e. }x\in Q_2 
\end{equation*}
And thus:
\begin{equation}\label{finaldecomp}
||B^*f||_{L_2(\omega,Q_2)}\leq ||m(B^*f,Q_2)||_{L_2(w,Q_2)}+D_6\cdot k \cdot||Mf||_{L_2(\omega,Q_2)}
\end{equation}
\begin{equation*}+D_6\cdot  k \cdot  \left|\left|\sum_{Q_1\in S^j(Q_2)}\fint_{Q_1}|f|\chi_{Q_1}\right|\right|_{L_2(\omega,Q_2)}
\end{equation*}
The first term of \ref{finaldecomp}  will be estimated in \ref{howmedianshouldbetreated} and the second term in \ref{lernerhasit}, thus finishing the proof of \ref{bistar}.
\begin{lemma}\label{howmedianshouldbetreated}
\begin{equation*}
||m(B^* f,Q_2)||_{L_2(w,Q_2)}\leq \beta k ||f||_{L_2(\omega,Q_2)} [\omega]_{A_2},
\end{equation*}
\end{lemma}
\begin{proof}
Decompose:
\begin{equation*}
B^* f=\sum_{Q\in S^{\prime}(Q_0), Q^*\subset Q_2}\frac{1}{|Q^*|}\int_{Q}f \cdot\chi_{Q^*}+\sum_{Q\in S^{\prime}(Q_0), Q_2\subsetneq Q^*}\frac{1}{|Q^*|}\int_{Q}f \cdot\chi_{Q^*}+\sum_{Q\in S^{\prime}(Q_0), Q^*\cap Q_2=\emptyset}\frac{1}{|Q^*|}\int_{Q}f \cdot\chi_{Q^*}.
\end{equation*}
Here the third term will be identically zero on $Q_2$ and the second term is zero due to \ref{hahahaempty}. Thus for $m(B^*f,Q_2)$ only the first term matters. Using $Q\subset Q^*$ we get:
\begin{equation*}
m(B^*f, Q_2)=m(B^*(f\cdot\chi_{Q_2}),Q_2)
\end{equation*}
By \ref{wtoweak} and \ref{discreteoperatornorm}:
\begin{equation*}
|m(B^*(f\cdot \chi_{Q_2}),Q_2)|\leq \frac{||B^*( f\cdot \chi_{Q_2})||_{L_{1,\infty}(Q_2)}}{|Q_2|}\leq \beta k \frac{||f||_{L_1(Q_2)}}{|Q_2|}
\end{equation*}
so that applying Holder's inequality and \ref{a2dyadic} we get
\begin{equation*}
||m(B^* f,Q_2)||_{L_2(w,Q_2)}\leq \beta k \frac{||f||_{L_{1}(Q_2)}}{|Q_2|}\left(\int_{Q_2}\omega\right)^{1/2}\leq \beta k ||f||_{L_2(\omega,Q_2)} [\omega]_{A_2}.
\end{equation*}
\end{proof}
\begin{lemma}
\label{lernerhasit}
\begin{equation*}
 \left|\left|\sum_{Q_1\in S^j(Q_2)}\left(\fint_{Q_1}|f| \right)\chi_{Q_1}(x)\right|\right|_{L_2(\omega,Q_2)}\leq D_7 [\omega]_{A_2}
\end{equation*}
\end{lemma}
We first note that for any function $g\in L_2(\omega^{-1})$,
\begin{equation*}
\left|\int_X \sum_{Q_1\in S^j(Q_2)} \fint |f|\chi_{Q_1}(x) \cdot g(x)dx\right|\leq 
\sum_{Q_1\in S^j(Q_2)}\fint_{Q_1}|f|\fint_{Q_1}|g|\cdot |Q_1|\leq 2\sum_{Q_1\in S^j(Q_2)}\fint_{Q_1}|f|\fint_{Q_1}|g|\cdot |E(Q_1)|
\end{equation*}
where for every $Q_1\in S^j(Q_2)$ we denoted \begin{equation*}E(Q_1)=Q_1\setminus \bigcup_{Q\in S(Q_2), Q \subsetneq Q_1}Q\end{equation*} and used \ref{half}.  Now let $\omega(Q) = \int_Q\omega$ and $M_{\omega^{-1}}(f) = \sup_Q\frac{1}{\omega(Q)}\int |f\omega|\omega^{-1}$.  We further estimate:
\begin{eqnarray*}
\sum_{Q_1\in S^j(Q_2)}\left( \fint_{Q_1}|f|\right)\left(\fint_{Q_1}|g| \right)|E(Q_1)|&\leq &
[\omega]_{A_2}\sum_{Q_1\in S^j(Q_2)}\left( \frac{1}{\omega^{-1}(Q_1)}\int_{Q_1}|f|\right) \left(\frac{1}{\omega(Q)}\int_{Q_1}|g| \right) |E(Q_1)|
\\&\leq & 
[\omega]_{A_2}\sum_{Q_1\in S^j(Q_2)} \int_{E(Q_1)} M_{\omega^{-1}}(f\omega)M_{\omega}(g \omega^{-1})
\\&\leq &
[\omega]_{A_2} \int_X M_{\omega^{-1}}(f\omega)M_{\omega}(g \omega^{-1})
\\&\leq &
[\omega]_{A_2} || M_{\omega^{-1}}(f\omega)||_{L_2(\omega^{-1})}||M_{\omega}(g\omega^{-1})||_{L_2(\omega)}
\\&\leq &
D_7[\omega]_{A_2} ||f||_{L_2(\omega)} ||g||_{L_2(\omega^{-1})}
\end{eqnarray*}
Where we use the fact that $M_{\omega}$ is bounded on $L^2(w)$ indepedent of $[\omega]_{A_2}$. 
Taking the suprenum over $||g||_{L^2(w^{-1})} = 1$,
\begin{equation*}
 \left|\left|\sum_{Q_1\in S^j(Q_2)}\left(\fint_{Q_1}|f| \right)\chi_{Q_1}(x)\right|\right|_{L_2(\omega,Q_2)}= \left|\left|\sum_{Q_1\in S^j(Q_2)}\left(\fint_{Q_1}|f| \right)\chi_{Q_1}(x)\right|\right|_{L_2(\omega)}\leq D_7 [\omega]_{A_2}
\end{equation*}
\end{proof}

\begin{lemma}\label{biitself}
\begin{equation*}
||B(f)||_{L_2(\omega)}\leq D_7 \cdot k [\omega]_{A_2} ||f||_{L_2(\omega)},
\end{equation*}
for $B=B_{k,j}$.
\end{lemma}
\begin{proof}
For any $g\in L_2(\omega^{-1})$ using \ref{bistar} we have:
\begin{equation*}
\int_X B(f)g=\int_X fB^*(g)=\int_X f\omega^{1/2}B^*(g)\omega^{-1/2}\leq ||f||_{L_2(\omega)} ||B^*g||_{L_2(\omega^{-1})}\leq
D_7 k [\omega^{-1}]_{A_2} ||f||_{L_2(\omega)}||g||_{L_2(\omega^{-1})}
\end{equation*}
so that the claim of the lemma and our main theorem follows.
\end{proof}
\section{Acknowledgements}
The first author is supported by a NSF graduate student fellowship.

\end{document}